\documentclass{article}
\usepackage[utf8]{inputenc}

\usepackage{amsfonts}
\usepackage{mathtools}
\usepackage{amsthm}
\usepackage{amssymb}
\usepackage{amsmath}
\usepackage{hyperref}
\usepackage{tikz-cd}

\usepackage[OT4]{fontenc}

\usepackage{graphicx}
\graphicspath{ {./images/} }

\theoremstyle{definition}
\newtheorem{Th}{Theorem}[section]
\newtheorem{Def}[Th]{Definition}

\newtheorem{Cor}[Th]{Corollary}
\newtheorem{Prop}[Th]{Proposition}
\newtheorem{Lem}[Th]{Lemma}
\newtheorem{Obs}[Th]{Observation}

\theoremstyle{remark}
\newtheorem{Notation}[Th]{Notation}
\newtheorem{Example}[Th]{Example}
\newtheorem{Problem}[Th]{Problem}
\newtheorem{Question}[Th]{Question}

\newtheorem{Rem}[Th]{Remark}

\usepackage{yfonts}

\newcommand{\C}{\mathbb{C}}
\newcommand{\Z}{\mathbb{Z}}

\newcommand{\F}{\mathcal{F}}

\newcommand{\op}{\operatorname}
\newcommand{\E}{\textswab{E}}

\title{Foliations and Galois Theory\\
in Positive Characteristic}
\author{Przemysław Grabowski}
\date{}

\begin{document}

\maketitle

\begin{abstract}
    We prove a Galois-type correspondence between compositions of purely inseparable field extensions (including infinite ones) and subalgebras of differential operators. This correspondence can be utilized to establish a connection between separable field extensions and purely inseparable field extensions. Specifically, it serves as a progression towards gaining a deeper comprehension of a foliation theory on varieties in positive characteristic.
\end{abstract}

\section{Introduction}
\paragraph{Geometric Motivation}
In algebraic geometry over $\C$, one can study foliations on normal varieties. A foliation is a subsheaf of a tangent bundle that is closed under Lie brackets. More details on this notion can be found in books \cite{Brunella} and \cite{Foliations in AlgGeo}. Foliations naturally arise as a generalisation of smooth dominant morphisms. These morphisms inject into foliations by taking their relative tangent bundles. This is an important injection for geometrical applications. However, an analog of this theory in positive characteristic, a theory of $1$-foliations introduced by Ekedahl in \cite{Ekedahl87}, fails at injecting smooth  dominant morphisms into $1$-foliations. Indeed, we have the following.

\begin{Example}
    Let $k$ be a perfect field of characteristic $p>0$. Let $f,g$ be morphisms $\mathbb{A}^2\to \mathbb{A}^1$ corresponding to ring maps $k[x]\to k[x,y]$ and $k[x+y^p]\to k[x,y]$ respectively. These morphisms are not equal, but their relative tangent bundles are equal. Indeed, each is spanned by $\frac{\partial}{\partial y}$.
\end{Example}

Although the operation is not injective, it is useful in a different way. It provides a correspondence between $1$-foliations and purely inseparable morphisms of height at most $1$, see the Proposition 2.4 in \cite{Ekedahl87}. 
It was successfully used to study positive characteristic geometry in many papers, e.g. \cite{Mi87}, \cite{Ekedahl88}, \cite{Mingmin}, and \cite{Langer1}. Nevertheless, the lack of the injectivity from smooth dominant morphisms is problematic, because it prevents us from studying foliations in a characteristic free way. 
This issue motivates the following question.

\begin{Question}\label{question0}
    Is it possible to extend $1$-foliation data to get a theory that admits an injection from smooth dominant morphisms to new extended data?
\end{Question}

The answer is affirmative. A solution is to use algebras of differential operators instead of modules of vector fields. And, as in the case of vector fields, we have to impose some extra conditions on these algebras. For the modules this is being a Lie algebra. For the algebras this will be being a \emph{$\infty$-foliation}.
We will define this notion in the Definition \ref{infty-foli} after stating our correspondence \ref{MainTheoremIntro}. This correspondence explains what information \emph{arbitrary} algebras of differential operators encode. 
After that, we will show that there is an injection from smooth dominant morphisms into $\infty$-foliations, the Proposition \ref{fibr into foli}. 
However, in this paper, we work on generic points only, i.e. with fields, to simplify our exposition and to connect our results with Galois theories. Therefore, we only prove an analog of such proposition for fields in positive characteristic.

Our main inspiration is a paper \cite{Ekedahl87} by Ekedahl.
According to our knowledge, Ekedahl's paper is the first one in which an idea of using differential operators instead of vector fields to study foliations in positive characteristic appears. However, he does not study arbitrary algebras of differential operators there, only some of them that he calls $m$-foliations on varieties. Moreover, dual objects of these subalgebras are formal groupoids. This observation has motivated a recent paper \cite{McQuillan} by McQuillan. In his paper, he proposes a characteristic free definition of a foliation in terms of formal groupoids. In the introduction, he states that his notion is a generalisation of Ekedahl's foliations.

\paragraph{Our Results}\label{results}
Let $K$ be a field of characteristic $p>0$. Let $k\coloneqq \bigcap_{n\ge 0} K^{p^n}$. We assume that $K/k$ is a finitely generated field extension.

In this paragraph, we state our Galois-type correspondence.
This result is related to purely inseparable Galois theories. These are theories establishing bijections between purely inseparable field extensions and some other structures. We quickly review some of them before the statement, because our correspondence partially overlaps with them.

Historically, the first one is a Jacobson Correspondence from \cite{Jacobson1} that covers only subfields $K\supset W \supset K^p$. This says that these subfields are in a bijection with $p$-Lie subalgebras of a tangent bundle $TK$. We recall it in \ref{Jacobson}. Later, in books \cite{Jacobson2}, \cite{Jacobson3}, the Jacobson Correspondence was extended to a Jacobson--Bourbaki Correspondence, we recall it in \ref{JB}. 
This is a bijection between arbitrary subfields of $K$ and closed $K$-subalgebras of additive operators on $K$, i.e. $\op{End}_\Z(K)$. 
This bijection restricted to subfields $W\subset K$ such that $K/W$ is purely inseparable maps these subfields onto finite $K$-subalgebras of differential operators $\op{Diff}(K)$ by the Lemma \ref{end=diff}. This could be called a Jacobson--Bourbaki purely inseparable Galois theory.

There are also other Galois theories that we would like to mention though they do not overlap with our correspondence. In a paper \cite{Sweedler}, Sweedler establishes a bijection between modular subfields $W\subset K$ and higher derivations. These subfields are purely inseparable subfields satisfying an extra condition. This bijection is compared with the Jacobson--Bourbaki Correspondence in Wechter's thesis \cite{Wechter}. Recently, in a paper \cite{PIGalois}, a new theory was proposed by Brantner and Waldron. Their bijection is between subfields of finite exponents $K\supset W \supset K^{p^r}$, $r\ge 0$, and partition Lie algebroids. These algebroids are a derived algebraic geometry generalisation of relative tangent bundles.

We have already observed that the Jacobson--Bourbaki purely inseparable Galois theory states that there is a bijection between subfields $W\subset K$ such that $K/W$ is purely inseparable and finite $K$-subalgebras $\mathcal{D}\subset \op{Diff}(K)$. This follows from the Jacobson--Bourbaki Correspondence, but this correspondence says nothing about infinite dimensional subalgebras. Indeed, these subalgebras are not closed, and their closures in $\op{End}_\Z(K)$ often coincide, see the Remark \ref{closures}. Nevertheless, it is possible to extend this purely inseparable Galois theory to cover all subalgebras by replacing purely inseparable subfields with towers of them. This is a part of our main theorem.

\begin{Th}[Theorem \ref{MainTheorem}]\label{MainTheoremIntro}
Let $K$ be a field of characteristic $p>0$. Let $k\coloneqq \bigcap_{n\ge 0} K^{p^n}$. We assume that $K/k$ is a finitely generated field extension.

There are natural bijections between the following data on $K$:
\begin{itemize}
    \item power towers on $K$: a tower of subfields $W_\bullet\subset K$, where $\bullet\ge 0$, such that $i\le j \Rightarrow W_j \cdot K^{p^i}=W_i$,
    \item Jacobson sequences on $K$: $(\F_1,\F_2,\ldots)$, see the Definition \ref{JacSeq},
    \item $K$-subalgebras of differential operators on $K$: $K\subset \mathcal{D}\subset \op{Diff}_k(K)$
\end{itemize}
Explicitly,
these bijections are given by the following operations:
\begin{itemize}
    \item an iterated tangent bundle: $W_\bullet \mapsto (T_{W_0/W_1},T_{W_1/W_2},T_{W_2/W_3}, \ldots)$,
    \item an iterated $\op{Ann}(\bullet)$: $(\F_1,\F_2,\ldots)\mapsto \op{Ann}(\F_\bullet)$,
    \item an algebra of relative differential operators: $W_\bullet \mapsto \op{Diff}_{W_\bullet}(K)\coloneqq \bigcup_{n\ge 0}\op{Diff}_{W_n}(K)$,
    \item an unpacking: $\mathcal{D} \mapsto \left(\F_1\coloneqq \mathcal{D}\cap TK,\ldots, \F_n\coloneqq d(K/W_{n-1})(\mathcal{D})\cap TW_{n-1}, \ldots \right)$.
\end{itemize}
Moreover, we have that ranks $\op{rk}(\F_i)\coloneqq \op{dim}_{W_{i-1}}(\F_i)$ are finite and nonincreasing: $\op{rk}(\F_i)\ge \op{rk}(\F_{i+1})$, and they satisfy $p^{\op{rk}(\F_i)}=\op{dim}_{W_{i+1}}(W_i)$.
\end{Th}

This theorem covers all $K$-subalgebras of differential operators. And, it does not impose any additional conditions on them such as being closed in a topology. 
In this correspondence, purely inseparable subfields are replaced by power towers on $K$.
However, an operation $W \mapsto W\cdot K^{p^\bullet}$ injects these subfields into power towers on $K$. The image consists of eventually constant power towers on $K$, see the Proposition \ref{FiniteLength}. This recovers the Jacobson--Bourbaki purely inseparable Galois theory from our theorem. Finally, the data of Jacobson sequences with an operation of unpacking is our generalisation of a comparison of the Jacobson Correspondence \ref{Jacobson} with the Jacobson--Bourbaki Correspondence \ref{JB}, see the Remark \ref{<L>capTK=L}.

Moreover, we can exploit the operation $W \mapsto W\cdot K^{p^\bullet}$ more. Indeed, this operation extends well to all subfields of $K$, but then it is not injective anymore. It is made precise in the following proposition.

\begin{Prop}[Proposition \ref{infty=s}]  
Let $W$ be a subfield of $K$. Let $W_\bullet\coloneqq W\cdot K^{p^\bullet}$ be its power tower on $K$, $\bullet\ge 0$, then power towers of $W$ and $\left(W\cdot k\right)^s$ on $K$ coincide, and $\bigcap_{\bullet \ge 0} W_\bullet =\left(W\cdot k\right)^s$, where $()^s$ is a separable closure in $K$.

In particular, an operation $W \mapsto W\cdot K^{p^\bullet}$ is injective on subfields $W\subset K$ such that $k\subset W \subset K$, and $W$ is separably closed in $K$.
\end{Prop}

An image of this operation is particularly well behaved if $K/W$ is a separable extension, see the Definition \ref{separable}. And, these extensions are an analog of smooth dominant morphisms. The image lands in $\infty$-foliations.

\begin{Def}\label{infty-foli}
    A $\infty$-foliation on $K$ is a power tower $W_\bullet$ on $K$ such that it is of a constant degree, i.e. $\op{dim}_{W_{i+1}}(W_i)=\op{dim}_{W_{i+2}}(W_{i+1})$ for all $i\ge 0$.
\end{Def}

In general, by the Theorem \ref{MainTheoremIntro}, we know that $\op{dim}_{W_{i+1}}(W_i)\ge \op{dim}_{W_{i+2}}(W_{i+1})$. So, these power towers are visibly distinguished among all power towers. Finally, the following is an analog of the injection of smooth dominant morphisms into foliations for fields in positive characteristic.

\begin{Prop}[Proposition \ref{fibr into foli}]
Let $W$ be a subfield of $K$ such that $k\subset W\subset K$ and $K/W$ is separable, then a power tower of $W$ on $K$ is a $\infty$-foliation on $K$.
\end{Prop}

\paragraph{Layout} The paper consists of five sections. The first one is the introduction. The second section \ref{section2} establishes a basic theory of power towers. The third section \ref{section3} recalls a construction of differential operators, and it proves some of their elementary properties. The emphasis is on positive characteristic structures. The fourth section \ref{section4} carefully proves the main theorem. 
In the last section \ref{section5}, we obtain some more results about power towers by using a theory established in the previous sections. In particular, we prove the injection of separably closed subfields into power towers, and we discuss a problem of extending $1$-foliations to $2$-foliations.

\paragraph{Acknowledgements}
This paper is a part of the author's PhD pursuit at the University of Amsterdam supervised by Diletta Martinelli and Mingmin Shen.

We would like to thank Piotr Achinger, Fabio Bernasconi, Federico Bongiorno, Lukas Brantner, Paolo Cascini, Zsolt Patakfalvi, Calum Spicer, Roberto Svaldi, and Joe Waldron for discussions about purely inseparable Galois theories, foliations, and related topics.

\newpage
\section{Power Towers}\label{section2}

In this section, we study power towers. We recall some preliminaries. The most important one is a notion of a $p$-basis \ref{Coordinates}. We define power towers and we prove basic facts about them that they are towers of extensions of exponents at most $1$ by \ref{Exponents1}, 
and that degrees of subsequent extensions are nonincreasing \ref{NonincreasingOrder}. After that, we move to comparing power towers with subfields. We define first integrals \ref{W.infty}. We show that subfields of finite exponents inject into power towers \ref{FiniteLength}. 
However, not all power towers are power towers of subfields. We give an example of this in \ref{NonAlgInt}. Finally, we introduce notions of $n$-foliations, and we show that power towers of separable subfields are $\infty$-foliation \ref{fibr into foli}. 

\paragraph{Preliminaries}
A basic theory of fields can be found in any good textbook on algebra, e.g. in Section 09FA on \cite{Stacks}, in \cite{Jacobson2}, or in \cite{Algebra2}. 

The following is a combination of 09HC, and 030K from \cite{Stacks}.

\begin{Lem}\label{separableclosure}
Let $L/K$ be an extension of fields. There is a maximal separable subextension $L/K^s/K$ such that $K^s/K$ is separable, where the maximality means that if $L/F/K$ is a subextension such that $F/K$ is separable, then $F \subset K^s$. Moreover, if $L/K$ is algebraic, then $L/K^s$ is purely inseparable.

The field $K^s$ is called a \emph{separable closure} of $K$ in $L$.
\end{Lem}

Here we fix notations about $p$-powers of fields.

\begin{Def}
Let $K$ be a field of characteristic $p>0$.
Let $n\ge 0$ be an integer.

We define a subfield of $p^n$-powers of $K$ by $K^{p^n}\coloneqq \{ x^{p^n}; x\in K \}$.

We define a supfield of $p^n$-roots of $K$ by $K^{1/p^n}\coloneqq \{ x^{1/p^n}; x\in K \}$, where the roots are taken in an algebraic closure of $K$.

Moreover, we define $K^{p^\infty}\coloneqq \bigcap_{n\ge 0} K^{p^n}$.

The field $K$ is called \emph{perfect} if $K=K^p$.
\end{Def}

In the literature, there are two equivalent common ways of addressing a quantification of purely inseparable extensions: a height in \cite{Algebra2}, and an exponent in \cite{Jacobson2}. We prefer using the word exponent. Here we recall what it means.

\begin{Def}
Let $K$ be a field of characteristic $p>0$.
Let $W$ be a subfield of $K$. Let $n\ge 0$ be an integer, or infinity.
Then the extension $K/W$ is of \emph{exponent $\le n$} if  $K\supset W\supset K^{p^n}$. And, it is of exponent $n$, if the $n$ is minimal such.
\end{Def}

Here is an example of purely inseparable extension that is not of finite exponent.

\begin{Example}
Let $k$ be a prefect field of characteristic $p>0$. Let $K=k(x_i)_{i\ge 1}$, and $W=k(x_i^{p^i})$. Then $K/W$ is purely inseparable, but the exponent is infinite.
\end{Example}

However, in good situations, all exponents of purely inseparable subfields are finite.

\begin{Lem}\label{finite exponents}
Let $K$ be a field of characteristic $p>0$ such that $K/K^{p^\infty}$ is a finitely generated field extension. Then any subfield $W\subset K$ such that $K/W$ is purely inseparable is of finite exponent.
\end{Lem}

\begin{proof}
    The extension $K/W$ is finitely generated, because it is a subextension of $K/K^{p^\infty}$ that is finitely generated by the assumption. Therefore, there are finitely many $x_i \in K$ such that $W(x_i)=K$. For each $i$ there is an integer $a_i$ such that $x_i^{p^{a_i}}\in W$. Let $A$ be the maximal number among them, then $W\supset K^{p^A}$. 
\end{proof}

The following is about existence of separating transcendence basis that we will need in some arguments. For the definitions of notions of a $p$-basis and of a $p$-independence see the Section 07P0 at \cite{Stacks}, of a separating transcendence basis and its existence see \cite{MacLane}, and see the Corollary A1.5 in \cite{Eisenbud} for the proof that they coincide in the case below.

\begin{Def}\label{Coordinates} 
Let $K$ be a field of characteristic $p>0$. Assume that $K/K^{p^\infty}$ is a finitely generated field extension. Then this extension admits a $p$-basis: elements $x_i$ such that $dx_i$ are a $K$-basis of $\Omega_{K/K^{p^\infty}}$,

and a separating transcendence basis: elements $x_i\in K$ that are algebraically independent over $K^{p^\infty}$ and $K^{p^\infty}(x_i)^s=K$. 

These notions coincide for $K/K^{p^\infty}$. 

Moreover, any set of $p$-independent elements for $K/K^{p^\infty}$: elements $x_i\in K$ such that $dx_i$ are $K$-linearly independent, can be completed to a $p$-basis.
\end{Def}

\paragraph{Basic Properties of Power Tower}

We assume that $K$ is a field of characteristic $p>0$ such that $K/k$ is a finitely generated field extension, where $k=K^{p^\infty}$.

\begin{Def}\label{powertower}
A \textbf{\emph{power tower}} on $K$ is a sequence of subfields $W_\bullet$ of $K$ such that $\bullet=0,1,2,\ldots$, and $W_j = W_i \cdot K^{p^j}$ for $j\le i$. 
\end{Def}

In the above definition, the dot sign stands for a composite of fields inside $K$. It is convenient to represent such a tower in a form of a diagram.

\begin{center}
    \begin{tikzcd}
    K=W_0 & W_1\ar[l] & K^p \ar[l]\\
    K \ar[u,equal]& W_2\ar[l]\ar[u] & K^{p^2}\ar[l]\ar[u] \\
    \ldots \ar[u,equal]& \ldots\ar[l]\ar[u] & \ldots\ar[l]\ar[u] \\
    K \ar[u,equal]& W_m\ar[l]\ar[u] & K^{p^m}\ar[l]\ar[u] \\
    \ldots \ar[u,equal]& \ldots\ar[l]\ar[u] & \ldots\ar[l]\ar[u] 
    \end{tikzcd}
\end{center}

On this diagram, we can notice two examples of power towers on $K$.

\begin{Example}\label{TrivialPowerTowers}
A constant tower $W_n\coloneqq K$ is a power tower on $K$. A tower of $p$-powers of $K$, namely $W_n\coloneqq K^{p^n}$, is a power tower on $K$.
\end{Example}

Two main properties of power towers are the following.

\begin{Lem}\label{Exponents1}
Let $W_\bullet$ be a power tower on $K$, then each extension $W_i/W_{i+1}$ is of exponent at most $1$, where $i\ge 0$. 
\end{Lem}
    
\begin{proof}
We have to proof that $W_{i+1} \supset W_i^p$. This follows from the following
\[
W_i^p={\left(W_{i+1}\cdot K^{p^i}\right)}^{p}=W_{i+1}^p\cdot K^{p^{i+1}}
\subset W_{i+1} \cdot K^{p^{i+1}}=W_{i+1},
\]
where, in the second equality, we used that taking $p$-powers commutes with taking composites. The last equality is from the Definition \ref{powertower}.
\end{proof}

\begin{Prop}[Nonincreasing Degrees]\label{NonincreasingOrder}
Let $W_\bullet$ be a power tower on $K$. It satisfies the following sequence of inequalities:
\[\op{dim}_{W_1}(K)\ge \op{dim}_{W_2}(W_1)\ge
\ldots\ge\op{dim}_{W_n}(W_{n-1})
\ge \ldots \ge 0.
\]
\end{Prop}

\begin{proof}
Let $n\ge 0$. We are going to show that $\op{dim}_{w_{n+1}}(W_{n})\ge \op{dim}_{w_{n+2}}(W_{n+1})$ by writing some generators of these extensions explicitly. Let $x_i\in K^{p^{n+1}}$, where $i=1,\ldots, r$, be a minimal subset such that $W_{n+1}=W_{n+2}(x_i)$. Notice that $x_i^p\in W_{n+2}$ by the Lemma \ref{Exponents1}. 
We claim that $\op{dim}_{W_{n+2}(x_i)}(W_{n+2}(x_i^{1/p}))=\op{dim}_{W_{n+2}}(W_{n+1})$. Indeed, if not, then there is a nontrivial $(W_{n+2}(x_i))$-linear relation $F$ between monomials $\prod_{i=1}^r {\left(x_i^{1/p}\right)}^{a_i}$, where $0\le a_i \le p-1$, but then $F^p$ is a $W_{n+2}$-linear relation for monomials $\prod_{i=1}^r {\left(x_i\right)}^{a_i}$ since $W_{n+2}(x_i)^p\subset W_{n+2}$. This contradicts the definition of $x_i$'s.

Now, we have a series of extensions 
\[
W_{n+2}\subset W_{n+1} \subset W_{n+1}(x_i^{1/p}) \subset W_{n}
\]
that gives us 
\begin{align*}
   \op{dim}_{W_{n+1}}(W_n)&=\op{dim}_{W_{n+1}(x_i^{1/p})}(W_n)\cdot \op{dim}_{W_{n+1}}(W_{n+1}(x_i^{1/p}))
   \\ & \ge \op{dim}_{W_{n+1}}(W_{n+1}(x_i^{1/p}))=\op{dim}_{w_{n+2}}(W_{n+1}). 
\end{align*}
\end{proof}

A special class of power towers are the ones that are eventually constant. They are characterized by having a finite length. 

\begin{Def}\label{Length}
We say that a power tower $W_\bullet$ on $K$ is of \emph{length at most $n$} if $W_N=W_n$ for $N\ge n$. It is of length $n$, if the $n$ is minimal such.
\end{Def}

A trivial observation is that a power tower $W_\bullet$ of length $n$ is determined by $W_n$. Indeed, if $i<n$, then $W_i=W_n\cdot K^{p^i}$ by the Definition \ref{powertower}, and if $i>n$, then $W_i=W_n$ by the Definition \ref{Length}. Consequently, we can conclude the following.

\begin{Cor}\label{FiniteLength}
Let $n$ be a positive integer. There is a bijection between subfields of $K$ of exponent $n$, and power towers on $K$ of length $n$. Explicitly, this is given by $W_\bullet\mapsto W_n$. $\qed$
\end{Cor}

In particular, the above says that subfields $W$ of $K$ such that $K/W$ is a purely inseparable extension are included in power towers. This is a special case of the Proposition \ref{infty=s}.

\paragraph{Power Towers of Subfields}\label{Subfields}
We introduce an operation from subfields of $K$ to power towers on $K$. It is one of main sources of power towers.

\begin{Def}\label{PTsubfield}
Let $W$ be a subfield of $K$. A power tower of $W$ on $K$ is defined by $W_\bullet \coloneqq W\cdot K^{p^\bullet}$, where $\bullet\ge 0$.

We say that $W_n$ is $W$ up to $p^n$-powers of $K$.
\end{Def}

We do a reality check.

\begin{Lem}
Let $W$ be a subfield of $K$. Then $W_\bullet$ is a power tower on $K$.
\end{Lem}
    
\begin{proof}
We have to show that $(W\cdot K^{p^n})\cdot K^{p^m}=W\cdot K^{p^m}$ for $0\le m\le n$. 
This follows from the fact that taking composite is associative. Indeed, $(W \cdot K^{p^n})\cdot K^{p^m}=W\cdot (K^{p^n}\cdot K^{p^m})=W\cdot K^{p^m}$
\end{proof}

For example, trivial power towers, the Example \ref{TrivialPowerTowers}, are power towers on $K$ of $K$, and $K^{p^{\infty}}$ respectively. And, power towers of subfields of finite exponent are power towers of finite length, the Corollary \ref{FiniteLength}. Here is another example.

\begin{Example}
Let $k$ be a perfect field of characteristic $p>0$.
Let $K=k(x,y)$, and $W=k(x)$. Then $W_n=k(x,y^{p^n})$. Moreover, we have $\op{dim}_{W_{n+1}}(W_{n})=p$ for every $n\ge 0$. Hence, it is a power tower of a constant degree.
\end{Example}

A name in the next definition is borrowed from a theory of foliations. In the theory of foliations, a foliation that is a relative tangent bundle of a morphism is called algebraically integrable, because its leaves are algebraic subvarieties, see  the Lemma 3.2.4 in \cite{Bongiorno}.

\begin{Def}
We say that a power tower $V_\bullet$ on $K$ is \emph{algebraically integrable} if there is a subfield $W$ of $K$ such that $V_\bullet=W_\bullet$, where $W_\bullet$ is a power tower of the subfield $W$ on $K$.
\end{Def}

By the Corollary \ref{FiniteLength}, we know that power towers of finite length are algebraically integrable, but there are power towers of infinite length that are not. This resembles a situation from a characteristic zero, where a question if a foliation is algebraically integrable is nontrivial.

\begin{Example}[Not Algebraically Integrable Power Tower]\label{NonAlgInt}
Let $k$ be a perfect field of characteristic $p>0$.
Let $K=k(x,y)$. We define
\[
W_n\coloneqq k(x+y^p+\ldots + y^{p^{n-1}},y^{p^{n}}).
\]
The tower $W_\bullet$ is a power tower on $K$, but it is not algebraically integrable. 

Indeed, if it was a power tower of a subfield $W$, then we would have $W\subset W\cdot k\subset \bigcap_{n\ge 0} W_n=k$, and, consequently, $W\cdot k=k$. But then, we would have that $W_n=W\cdot K^{p^n}=K^{p^n}$. It is a contradiction.

We provide more examples like this one in the Example \ref{examples}.
\end{Example}

In the above example, we used an object $\bigcap_{n\ge 0} W_n$. This is analogical to \emph{first integrals} from a theory of foliations, so we are going to call it such as well. In general, first integrals are functions on a space that are constant along solutions of a system of differential equations on that space, see \cite{firstintegrals}.

\begin{Def}\label{W.infty}
Let $W_\bullet$ be a power tower on $K$. We define a field 
\[
W_\infty\coloneqq \bigcap_{n\ge 0} W_n.
\] 
We call it a field of \emph{first integrals} of the power tower $W_\bullet$.
\end{Def}

First integrals make a recovery of a subfield $W$ from its power tower on $K$ possible if $W$ contains $k$, and $W$ is separably closed in $K$. We prove it later in the Proposition \ref{infty=s}, because our proof uses some facts about differential operators from the next section \ref{section3}.

\paragraph{$n$-Foliations}
In a paper \cite{Ekedahl87}, Ekedahl introduced a notion of a $n$-foliation. These are subalgebras of differential operators on a variety satisfying an extra condition. If one transports his definition verbatim to the case of fields, and then uses our correspondence \ref{MainTheorem} to get a power tower, then they would end up with the following notion of a $n$-foliation for fields.

\begin{Def}\label{n-foliation}
Let $0\le n \le \infty$.
Let $W_\bullet$ be a power tower on $K$.  It is called a \emph{$n$-foliation} if it is of length $n$, and it satisfies $\op{dim}_{W_{i+1}}(W_{i})=\op{dim}_{W_{j+1}}(W_j)$ for $0\le i,j< n$. 
    
A rank of a $n$-foliation $W_\bullet$ is  $\op{log}_p(\op{dim}_{W_{1}}(K))$.
\end{Def}

\paragraph{Fibrations Into $\infty$-Foliations}

We prove a positive characteristic version of an injection of smooth dominant morphisms into foliations for fields. 
We label it an \textit{injection of fibrations into $\infty$-foliations}.
We start with an example.

\begin{Example}\label{transcedental => infty-foli}
Let $K$ be a field of characteristic $p$ such that $K/k$ is a finitely generated field extension, where $k=K^{p^\infty}$.

Let $K\supset W\supset k$ be both purely transcendental extensions. 
Then, a power tower of $W$ on $K$ is a $\infty$-foliation. Indeed, we can write $W=k(x_1,\ldots,x_r)$ and $K=W(y_1,\ldots y_m)$, where $x_i$ and $y_i$ are transcendental bases for $W/k$ and $K/W$ respectively. 
Then 
$W_n=W(y_1^{p^n},\ldots y_m^{p^n})=k(x_1,\ldots x_r,y_1^{p^n},\ldots y_m^{p^n})$.
So, $\op{dim}_{W_n}(W_{n-1})=p^m$ is constant. Therefore, $W_\bullet$ is a $\infty$-foliation on $K$ of rank $m$.
\end{Example}

The above example is a special case of the injection. Before stating the general case, we recall the definition of (not finite) separable extensions of fields, see 030O at \cite{Stacks}.

\begin{Def}\label{separable}
    A finitely generated extension of fields $A/B$ is separable if it admits a separating transcendence basis, i.e. there are algebraically independent over $B$ elements $x_i\in A$ such that $A$ is a separable closure of $B(x_i)$ in $A$.
\end{Def}

Here is the statement of the injection.

\begin{Prop}[Fibrations Inject Into $\infty$-Foliations]\label{fibr into foli}
Let $K$ be a field of characteristic $p>0$ such that $K/k$ is a finitely generated field extension, where $k=K^{p^\infty}$.

A power tower of a subfield $k\subset W\subset K$ on $K$ such that $K/W$ is separable is a $\infty$-foliation on $K$.
\end{Prop}

\begin{proof}
We can assume that $W$ is separably closed in $K$ since power towers of $W$ and $W^s$ coincide.

Let $y_i$ be a $p$-basis of $W$ over $k$, see the Definition \ref{Coordinates}. Let $x_i$ be a separating transcendence basis of $K$ over $W$. Then $x_i,y_i$ are a $p$-basis of $K$ over $k$.
In particular, we have that $k(x_i)^s=W$, $W(y_j)^s=K$, and $k(x_i,y_j)^s =K$. 

From the Example \ref{transcedental => infty-foli}, we know that a power tower of $k(x_i)$ on $k(x_i,y_j)$ is a $\infty$-foliation. The rest of the proof is just pulling back this information to a power tower of $W$ on $K$. 

Let $n\ge 0$ be an integer. 
We have that $k(x_i, y_j^{p^n})^s=W_n$, where the separable closure is taken in $K$. 
Moreover, we claim that
\[
    k(x_i, y_j^{p^{n}})\otimes_{k(x_i, y_j^{p^{n+1}})} W_{n+1} \simeq W_{n+1}(y_j^{p^n}) = W_{n}
\]
via natural maps to $K$. 
Of course, we have that $W_{n+1}(y_j^{p^n}) \subset W_{n}$.
We also have $W_{n+1}(y_j^{p^n})=W_{n+1}(y_j^{p^n})^s\supset k(x_i, y_j^{p^n})^s=W_n$. Therefore $ W_{n+1}(y_j^{p^n}) = W_{n}$.

Finally, we can conclude that a dimension of $ k(x_i, y_j^{p^{n}})$ over $ k(x_i, y_j^{p^{n+1}})$ is the same as a dimension of $W_n$ over $W_{n+1}$, because the later comes from the former by extending coefficients to $W_{n+1}$ and that operation is flat.
\end{proof}

\newpage
\section{Differential Operators}\label{section3}

This section is devoted to recalling differential operators in positive characteristic. First, we define these operators \ref{diffoper}, describe a filtration by orders, and embed them into additive operators \ref{embedding}. This makes a set of all of differential operators an algebra. This algebra admits a second filtration that we call a canonical filtration by finite subalgebras \ref{canonical}. After that, we demonstrate that any field extension induces a surjective differential between algebras of differential operators \ref{differential}. This extends a well known differential between tangent spaces. Moreover, we recall how to write down a differential operator in coordinates \ref{formulas}, and how these formulas behave with respect to differentials between $p$-powers \ref{symbol.differential}.

\paragraph{Construction}
This paragraph is based mostly on sections 16.3 and 16.8 from \cite{16.8}. 

\begin{Def}\label{diffoper}
Let $K/W$ be an extension of fields. Let $p_1,p_2: K\to K\otimes_W K$ be projections. They are $a\mapsto a\otimes 1, 1\otimes a$ respectively. 
We consider $K\otimes_W K$ as a $K$-algebra via $p_1$.
Let $\Delta: K\otimes_W K\to K$ be a diagonal map. It is given by $a\otimes b\mapsto ab$. Let $J_{K/W}$ be the kernel of $\Delta$. Let $n\ge 0$ be an integer.

For the extension, we define its space of \emph{principal parts} of order at most $n$ by $\mathcal{P}^{\le n}_{K/W}\coloneqq K \otimes_W K/ J^{n+1}_{K/W}$.

We define a set of differential operators of order at most $n$ on $K$ relative to $W$ by $\op{Diff}^{\le n}_{W}(K)\coloneqq \op{Hom}_K(\mathcal{P}^{\le n}_{K/W}, K)$.

A set of all \emph{differential operators} on $K$ relative to $W$ is defined by $\op{Diff}_{W}(K)\coloneqq \bigcup_{n\ge 0} \op{Diff}^{\le n}_{W}(K)$, where the identifications come from natural surjections of spaces of principal parts.
\end{Def}

Here are some formal observations following from the construction.

\begin{Def}\label{Ekedahl algebra}
The tensor product $K\otimes_W K$ decomposes into $K \oplus J_{K/W}$ by the choice of $p_1$, i.e. it is an augmented $K$-algebra. This implies that $\op{Diff}_{W}(K)$ decomposes too. We denote this decomposition by $K \oplus \E(K/W)$.
\end{Def}

In the above, we can observe that $\E(K/W)\cap \op{Diff}^{\le 1}_{W}(K)$ is $T_{K/W}$.

\begin{Obs}\label{embedding}
By using $p_2$, we can embed the set $\op{Diff}_{W}(K)$ into $\op{End}_W(K)$. This map is given by 
\[
\phi: \mathcal{P}^{\le n}_{K/W} \to K \mapsto D_\phi\coloneqq \phi \circ p_2: K\to K\otimes_W K \to \mathcal{P}^{\le n}_{K/W} \to K,
\]
where $n$ is an integer. It is an injection, and the image is closed under composition. This makes the set of differential operators an algebra of operators on $K$. It is covered in the Propositions 16.8.4 and 16.8.9 in \cite{16.8}. Consequently, the algebra of differential operators is a filtered algebra in a natural way. The filtration comes from subsets of order at most $n$. We call this filtration an \emph{order filtration}.

Moreover, the image of $\op{Diff}^{\le n}_{W}(K)$ is precisely a set of operators $D$ such that for any $n+1$ elements $a_0,a_1,\ldots,a_n\in K$ we have $[a_0,[a_1,[\ldots,[a_n,D]\ldots]=0$, where $[a,b]=a\circ b-b\circ a$. And, the image of $\E(K/W)$ are differential operators $D$ such that $D(1)=0$.
\end{Obs}

The order filtration is important, and any mention of a graded algebra of differential operators is exactly
the graded algebra associated to this filtration.

\paragraph{Canonical Filtration}
We show that in positive characteristic an algebra of differential operators admits another filtration. The new one is given by finite subalgebras corresponding to a power tower of $K^{p^{\infty}}$ on $K$. We call it a canonical filtration. It is a well known structure, see for example \cite{Smith}, but it is usually written with $\op{End}$ instead of $\op{Diff}$. Anyway, our Lemma \ref{end=diff} says that both ways are identical.

Let $K$ be a field of characteristic $p>0$ such that $K/k$ is a finitely generated field extension, where $k= K^{p^\infty}$.

\begin{Def}\label{canonical}
We define $\op{Diff}(K)\coloneqq \op{Diff}_k(K)$, and $\op{Diff}_{p^n}(K)\coloneqq \op{Diff}_{K^{p^n}}(K)$.

We call a set of $\op{Diff}_{p^\bullet}(K)$, where $\bullet \ge 0$,  a \emph{canonical filtration} of $\op{Diff}(K)$.
\end{Def}

We start with a useful observation.

\begin{Lem}\label{end=diff}
Let $K/W$ be a subfield of finite exponent, then we have $\op{Diff}_W(K)=\op{End}_W(K)$.
\end{Lem}

\begin{proof}
We know that one is a subalgebra of the other by \ref{embedding}. And, we can compute from the construction that they have the same finite dimensions over $K$. Therefore, they are equal.
\end{proof}

\begin{Rem}
The Lemma \ref{end=diff} is not true if a subfield $K/W/k$ is not of finite exponent.
\end{Rem}

We check that the canonical filtration is a filtration.

\begin{Prop}
Algebras $\op{Diff}_{p^\bullet}(K)$ are subalgebras of $\op{Diff}(K)$, and we have $\op{Diff}_{p^{n}}(K)\subset\op{Diff}_{p^{n+1}}(K)$ for $n\ge 0$, and $\op{Diff}(K)=\bigcup_{n\ge 0} \op{Diff}_{p^n}(K)$.
\end{Prop}

\begin{proof}
    These algebras are subalgebras from the construction of the embedding in the Observation \ref{embedding}. The inclusions are trivial from the Lemma \ref{end=diff}.

    Let $J_{p^n}$ be an ideal in $K\otimes_k K$ generated by $p^n$-powers of elements from $J_{K/k}$, then $K\otimes_{K^{p^n}} K=K\otimes_k K/J_{p^n}$. So, it is enough to compare the ideal $J_{p^n}$ with powers of the ideal $J_{K/k}$. For example, we have that $J_{K/k}^{p^{n-1}}\supset J_{p^n}$. Therefore, we have
\[
\op{Diff}_{k}^{\le p^{n-1}}(K) \subset \op{Diff}_{p^n}(K).
\] 
Hence, the ascending sum is everything.
\end{proof}

A corollary of the above proposition is an infamous fact that differential operators in positive characteristic are not finitely generated. In characteristic zero, the finite generation is true. Even more, they are generated in degree $1$ by derivations. 

\begin{Cor}
A $K$-algebra $\op{Diff}(K)$ is not finitely generated unless $K=k$.
\end{Cor}

\begin{proof}
    If $K\ne k$, then each extension $K/K^{p^n}$ is nontrivial, of finite dimension, and the dimensions grow with $n$. Hence, subalgebras $\op{Diff}_{p^n}(K)$ are finite and the inclusions between them are proper. Assume that $\op{Diff}(K)$ is finitely generated by elements $D_i$, then each of these operators has a finite order. In particular, each $D_i$ belongs to a subalgebra $\op{Diff}_{p^n}(K)$ for some $n$. This means that there must be a single $N$ such that all $D_i$ are in $\op{Diff}_{p^N}(K)$, but this subset is a subalgebra, therefore the subalgebra generated by all $D_i$ would be a subalgebra of $\op{Diff}_{p^N}(K)$. This is a contradiction, because this subalgebra is $\op{Diff}(K)$.
\end{proof}

\paragraph{Explicit Formulas via Coordinates}\label{formulas}
We establish calculus-like formulas for writing down differential operators in coordinates.

Let $K$ be a field of characteristic $p>0$ such that $K/k$ is a finitely generated field extension, where $k=K^{p^\infty}$.

We know from the Definition \ref{Coordinates} that $K/k$ admits a $p$-basis. Let $x_1,x_2,\ldots, x_N$ be a $p$-basis of $K/k$, where $N$ is a transcendence degree of $K$ over $k$. Then (16.11.1) from \cite{16.8} says that  monomials in $dx_i\coloneqq p_2(x_i)-p_1(x_i)$ that are nonzero in $\mathcal{P}_{K/k}^{n}$ are a basis of $\mathcal{P}_{K/k}^{n}$ (the empty monomial $1$ is included). We write these monomials using a multi-index notation $d(x)^e\coloneqq\prod d(x_i)^{e_i}$, where $e_i \ge 0$. We denote a dual basis by $D^e$. These can be identified with classical symbols.

\begin{Notation}\label{symbol}
We define a symbol $\frac{1}{a!}\frac{\partial^{a}}{\partial x_i ^{a}}$ to be $D^e$ for $e=(0,\ldots,0, a, 0,\ldots,0)$, where $a\ge 0$ is at the $i$-th position. 
\end{Notation}

\begin{Lem}[(16.11.1.7) in \cite{16.8}]\label{obvious.values}
We have that  $\frac{1}{a!}\frac{\partial^{a}}{\partial x_i ^{a}}(x_i^b)={{b}\choose{a}} x_i^{b-a}$, 
and $\frac{1}{a!}\frac{\partial^{a}}{\partial x_i ^{a}}(x_j^b)=0$ 
for $i\ne j$, $a,b\ge 0$. $\qed$
\end{Lem}

There is a unique way to express any differential operator in terms of the above symbols.

\begin{Cor}\label{unique expression}
Let $D \in \op{Diff}(K)$ be a differential operator. Then there are unique coefficients $A_{a_1,\ldots,a_N}\in K$, almost all equal zero, such that 
\[D=\sum_{a_i\ge 0}\left(A_{a_1,\ldots,a_N}\prod_{i=1}^{N} \frac{1}{a_i!}\frac{\partial^{a_i}}{\partial x_i ^{a_i}}\right).
\]
Moreover, we have $D(\prod_{i=1}^{N} x_i^{a_i})=A_{a_1,\ldots, a_N}$. Therefore, differential operators are determined by their values on monomials $\prod_{i=1}^{N} x_i^{a_i}$.

Furthermore, elements from $\op{Diff}_{p^n}(K)$ are determined by a property: $D\in\op{Diff}_{p^n}(K)$ if and only if $A_{a_1,\ldots,a_N}=0$ if there is an $i$ such that $a_i\ge p^n$.
\end{Cor}

\begin{proof}
    Let $D$ be a differential operator on $K$ of order at most $M$. It can be uniquely written as a $K$-linear combination of operators $D^e$, where $\sum |e_i| \le M$.
    Now, by the Lemma \ref{obvious.values}, we can compute that $D^e=\prod_i\frac{1}{e_i!}\frac{\partial^{e_i}}{\partial x_i^{e_i}}$. This finishes the proof of the unique representation.

    The last part follows from the explicit description of $K\otimes_{K^{p^n}} K$ in terms of monomials $dx^e$. Indeed, it is $K$-spanned by monomials with $e_i<p^n$ for all $i$.
\end{proof}

\begin{Rem}
   We can reduce a number of symbols needed to write any operator as a polynomial in them. Indeed, if we have a $p$-adic expansion of a natural number $a=a_0+a_1 p+a_2 p^2 + \ldots$, then an operator $\prod_{m\ge 0}\frac{1}{p^m!}\frac{\partial^{p^m}}{\partial x_i ^{p^m}}^{\circ a_m}$ differs from $\frac{1}{a!}\frac{\partial^{a}}{\partial x_i ^{a}}$ by a integral scalar that is invertible modulo $p$. This means that the symbols $\frac{1}{p^m!}\frac{\partial^{p^m}}{\partial x_i ^{p^m}}$ generate the algebra of differential operators. And, the ones with $m<n$ generate $\op{Diff}_{p^n}(K)$.
\end{Rem}

\paragraph{Differential between Algebras of Differential Operators}

It is well known that for any map between two spaces $f:X\to Y$, we get a differential between their tangent spaces $df: TX\to TY$. It turns out that $df$ naturally extends to a map $df: \op{Diff}(X)\to \op{Diff}(Y)$ between their spaces of differential operators. We present how it works for fields. An original construction of $df$ is covered around the Proposition (16.4.18) in \cite{16.8}.

Let $K$ be a field of characteristic $p>0$ such that $K/k$ is a finitely generated field extension, where $k=K^{p^\infty}$.

Let $k\subset W \subset K$ be a subfield. This induces a map:
\[
W\otimes_k W\to K\otimes_k K,
\]
that, for every $n\ge 0$, induces a map
\[
\mathcal{P}^{\le n}_{W/k}=W\otimes_k W/J_{W/k}^{n+1}\xrightarrow{\Delta^{\le n}_{K/W}} K\otimes_k K/J_{K/k}^{n+1}=\mathcal{P}^{\le n}_{K/k}.
\] 
We can use these to define differentials. Indeed, we have a series of $K$-linear maps:
\[
        \op{Hom}_K (\mathcal{P}^{\le n}_{K/k}, K) \rightarrow
        \op{Hom}_K (\mathcal{P}^{\le n}_{W/k}\otimes K, K) \leftarrow
        \op{Hom}_W (\mathcal{P}^{\le n}_{W/k}, K) \leftarrow
        \op{Hom}_W (\mathcal{P}^{\le n}_{W/k}, W) \otimes K,
\]
that are given respectively by $\phi \mapsto \phi \circ (\mathcal{P}^{\le n}_{W/k}\otimes K\to \mathcal{P}^{\le n}_{K/k}); \
\phi \mapsto (a \otimes x \mapsto x\phi(a)); \
\phi \otimes y \mapsto ((W \to K) \circ \phi) \cdot y$. We can observe that the arrows pointing left are injections between spaces of the same dimensions, therefore they are isomorphisms.
Now, we compose the arrow pointing right with the inverses of the other arrows to obtain a differential $d_{K/W}^{\le n}$. These are compatible with each other, thus they define a single differential on the whole algebra of differential operators. 

Here is a key theorem about these differentials.

\begin{Prop}\label{differential}
Let $K$ be a field of characteristic $p>0$ such that $K/k$ is a finitely generated extension of fields, where $k=K^{p^\infty}$. Let $k\subset W \subset K$ be a subfield.
The inclusion $W\subset K$ induces a differential:
\[
d(K/W): \op{Diff}(K) \to \op{Diff}(W)\otimes K.
\]
Explicitly, for a given differential operator $D$, its image $d(K/W)(D)$ is an operator from $W$ to $K$ given by restricting $D$ to $W$. Moreover, this differential is surjective.
\end{Prop}

\begin{proof}
The construction of the differential is done above. The restriction part follows from the explicit maps for the $K$-linear maps in the construction. So, the only thing to prove is the surjectivity.

We prove the surjectivity in three steps. Doing so, without loss of generality, we can assume that $W^s=W$, because we have $\op{Diff}(W)\otimes K=\op{Diff}(W^s)\otimes K$ since any $p$-basis of $W/k$ is a $p$-basis of $W^s/k$ and any $p$-basis determines all differential operators by the Corollary \ref{unique expression}.

 First, we assume that there is an integer $n$ such that $W\supset K^{p^n}$. Then we can filter algebras in the following way:
\[
\op{Diff}_{K^{p^n}}(K)\subset \op{Diff}_{K^{p^{n+1}}}(K)\subset \op{Diff}_{K^{p^{n+2}}}(K) \subset \ldots \subset \op{Diff}(K)
\]
\[
\op{Diff}_{K^{p^n}}(W)\subset \op{Diff}_{K^{p^{n+1}}}(W)\subset \op{Diff}_{K^{p^{n+2}}}(W) \subset \ldots \subset \op{Diff}(W)
\]
From the fact that the differential is a restriction, we get that the differential is compatible with the above filtrations. Indeed, any $K^{p^{n+\bullet}}$-linear map on $K$ must be mapped to $K^{p^{n+\bullet}}$-linear map on $W$. Hence, it is enough to show that all maps 
$\op{Diff}_{K^{p^{n+\bullet}}}(K)\to \op{Diff}_{K^{p^{n+\bullet}}}(W)\otimes K$ are surjective. However, by the Lemma \ref{end=diff}, this is equivalent to showing that a restriction map
\[
\op{End}_{K^{p^{n+\bullet}}}(K)=\op{Hom}_{K^{p^{n+\bullet}}}(K,K)\to \op{End}_{K^{p^{n+\bullet}}}(W)\otimes K=\op{Hom}_{K^{p^{n+\bullet}}}(W,K)
\]
is surjective. This is trivial. Indeed, let $\phi\in \op{Hom}_{K^{p^{n+\bullet}}}(W,K)$. We have that $W\subset K$ is a $K^{p^{n+\bullet}}$-linear subspace, so there exists a complementary vector space $V$ such that $W\oplus V= K$. Let $\psi\in \op{Hom}_{K^{p^{n+\bullet}}}(K,K)$ be an operator that is equal $\phi$ on $W$ and $0$ on $V$, then a restriction to $W$ of $\psi$ is equal $\phi$. Therefore, the map is surjective.

Second, we assume that $W\subset K$ satisfies $W^{1/p^{\infty}}\cap K=W$, where $W^{1/p^{\infty}}$, and the intersection are taken in a perfection of $K$. This means that $W^{1/p^{\infty}}$ and $K$ are linearly disjoint over $W$. Therefore, by the Theorem A1.3 from \cite{Eisenbud}, we have that $K/W$ is separable, the Definition \ref{separable}. This means that there are $p$-bases $x_i$ and  $y_j$ for $W/k$ and $K/W$ respectively. Moreover, together, they are a $p$-basis for $K/k$, the Definition \ref{Coordinates}. 
Indeed, $x_i, y_j$ are algebraically independent over $k$ and $k(x_i,y_j)^s=\left(k(x_i)^s(y_j)\right)^s=W(y_j)^s=K$. From the Corollary \ref{unique expression}, we know that differential operators are determined by values on monomials from a $p$-basis. Consequently, let $\phi\in \op{Diff}(W)$. There is a unique operator $\psi\in\op{Diff}(K)$ such that $\psi=\phi$ on monomials in $x_i$ and zero on any other monomial in $x_i, y_j$. This means that $d(K/W)(\psi)=\phi \otimes 1$. Hence, the map $d(K/W)$ is surjective, because any element of $\op{Diff}(W)\otimes K$ is a $K$-linear combination of elements of the form $\phi\otimes 1$.

The last step is a combination of the previous two. Let $k\subset W\subset K$ be a subfield. We define $W'\coloneqq W^{1/p^\infty}\cap K$ \footnote{$W'$ could be called a ``perfect saturation'' of $W$ in $K$.}. An extension $W'/W$ satisfies the first step by the Lemma \ref{finite exponents}. 
And, an extension $K/W'$ satisfies the second step, therefore the differential $d(K/W)$ is a composition of two surjections:
\[
\op{Diff}(K)\to \op{Diff}(W')\otimes K \to \left(\op{Diff}(W) \otimes W'\right) \otimes K.
\]
This finishes the proof.
\end{proof}

Here are some simple corollaries.

\begin{Cor}\label{kernel}
   Let $K,W$ be like in \ref{differential}. 
The kernel of $d(K/W)$ is equal $\op{Diff}(K) \cdot \E(K/W)$, i.e. it is a left ideal generated by $\E(K/W)$.
\end{Cor}

\begin{proof}
Being in the kernel means to be zero when restricted to $W$. This property is preserved under composition with anything from the left. Moreover, $\E(K/W)$ is in the kernel. Finally, the last fact that $\op{Diff}(K) \cdot \E(K/W)$ is the whole kernel is just dimension counting order by order by writing down explicit formulas for them in terms of diagonal ideals.
\end{proof}

\begin{Cor}\label{short exact sequence for diff}
    Let $K,W$ be like in \ref{differential}. Then we have a short exact sequence of $K$-vector spaces:
    $$0\to \op{Diff}(K) \cdot \E(K/W) \to \op{Diff}(K) \to \op{Diff} (W) \otimes K\to 0.$$
\end{Cor}

\begin{proof}
    This is the surjectivity from \ref{differential}, and the kernel description from \ref{kernel}.
\end{proof}

We finish by providing explicit formulas for how differentials act on the symbols \ref{symbol}.

\begin{Prop}\label{symbol.differential}
Let $K$ be like in \ref{differential}. Let $x_1,\ldots x_n$ be a $p$-basis for $K/k$, then $x_1^{p^m},\ldots x_n^{p^m}$ is a $p$-basis for $K^{p^m}/k$. Moreover, we have that
\[
    d(K/K^p)\left(\frac{1}{p^{m}!}\frac{\partial^{p^m}}{\partial x_i ^{p^m}}\right)= \frac{1}{p^{m-1}!}\frac{\partial^{p^{m-1}}}{\partial \left(x_i^p\right)^{p^{m-1}}}.
\]
Consequently,
\[
    d(K/K^{p^m})\left(\frac{1}{p^{m}!}\frac{\partial^{p^m}}{\partial x_i ^{p^m}}\right)= \frac{\partial}{\partial \left(x_i^{p^m}\right)}.
\]
\end{Prop}

\begin{proof}
The fact about $p$-bases is obvious from the perspective that they are also separating transcendental bases, the Definition \ref{Coordinates}.

The rest follows from comparing values of the symbols on monomials in these $p$-bases by the Corollary \ref{unique expression}.
\end{proof}

Consequently, we could use notations $\frac{1}{p^{m}!}\frac{\partial^{p^m}}{\partial x_i ^{p^m}}$ and $\frac{\partial}{\partial \left(x_i^{p^m}\right)}$ interchangeably on any subfield $K\supset W\supset k$ by identifying them by differentials.

\newpage
\section{Correspondence}\label{section4}

This section proves our correspondence \ref{MainTheorem}. We start with recalling Jacobson--Bourbaki \ref{JB} and Jacobson \ref{Jacobson} Correspondences that partially overlap with our theorem. Then, we iterate an operation of relative tangent bundle to obtain an iterated Jacobson Correspondence \ref{part2}. We conclude a correspondence between power towers and subalgebras by restricting a canonical filtration \ref{canonical} to all subalgebras of differential operators and then using a corollary \ref{effective} of a discrete Jacobson--Bourbaki correspondence \ref{discrete}. 
After that, we introduce an operation of unpacking \ref{part3} that is defined in terms of differentials between algebras of differential operators \ref{differential}. Its role is to compare subalgebras with $p$-Lie algebras. Finally, we amalgamate all the above with the Proposition \ref{NonincreasingOrder} into the Theorem \ref{MainTheorem}.

\paragraph{Jacobson Theorems}

We recall two Jacobson--Bourbaki Correspondences, and a corollary called a Jacobson Correspondence.

The first one is a general solution to a Galois problem that is describing all subfields of a given field in terms of a structure. We call it topological, because it involves a topology. More details can be found in the Chapter 8 of \cite{Jacobson3}.

\begin{Th}[Topological Jacobson--Bourbaki Correspondence]\label{JB}
Let $L$ be a field. There is a natural order reversing correspondence between subfields $K$ of $L$, and closed in a finite topology $L$-subalgebras of additive endomorphisms of $L$. 

Explicitly,
this correspondence is given by the following formulas:
\begin{align*}
   K\subset L \mapsto& \mathcal{JB}(K)\coloneqq \op{End}_K(L)\subset \op{End}_\Z(L)\\
   L\subset \mathcal{JB}=\overline{\mathcal{JB}}\subset \op{End}_\Z(L) \mapsto& \op{const}(\mathcal{JB})\coloneqq \{x\in L; \forall_{D\in \mathcal{JB}} [x,D]=0\}\subset L.
\end{align*}
\end{Th}

The second one is called discrete by us, because by restricting ourselves to subfields of finite corank the above topological assumption becomes empty. It is proved directly in \cite{Jacobson2} on page 22.

\begin{Th}[Discrete Jacobson--Bourbaki Correspondence]\label{discrete}
Let $L$ be a field. There is a natural order reversing correspondence between subfields $K$ of $L$ such that $L/K$ is finite, and finite dimensional $L$-subalgebras of $\op{End}_\Z(L)$.

Explicitly,
this correspondence is given by the following formulas:
\begin{align*}
   K\subset L\mapsto& \mathcal{JB}(K)\subset \op{End}_\Z(L)\\
   L\subset \mathcal{JB}\subset \op{End}_\Z(L) \mapsto& \op{const}(\mathcal{JB})\subset L.
\end{align*}
Moreover, we have that $\op{dim}_K(L)=\op{dim}_L(\mathcal{JB}(K))$.
\end{Th}

Finally, one can deduce a purely inseparable Galois theory of exponent at most one from the discrete Jacobson--Bourbaki Correspondence. It was originally proved in \cite{Jacobson1}, and later it was written down as a corollary from the discrete Jacobson--Bourbaki correspondence in the Theorem 8.47 of \cite{Jacobson3}

\begin{Th}[Jacobson Correspondence]\label{Jacobson}
Let $K$ be a field of characteristic $p>0$ such that $K/K^p$ is a finite extension.

There is an inclusion reversing correspondence between subfields $K^p\subset W\subset K$ and $p$-Lie algebras on $K$. 

Explicitly, it is given by
\begin{align*}
    K \supset W \supset K^{p} &\mapsto T_{K/W} \subset TK,\\
    \F \subset TK &\mapsto K \supset \op{Ann}(\F) \supset K^p,
\end{align*}
where $\op{Ann}(\F)\coloneqq\{x\in K| \forall_{D\in \F} D(x)=0\}$.
Moreover, $p^{\op{dim}_K(\F)}= \op{dim}_W(K)$.
\end{Th}

We remark on a fact from a prove of the above.

\begin{Rem}\label{<L>capTK=L}
    A part of the proof of the Jacobson Correspondence is showing that a $K$-subspace $\F$ of $TK$ is a $p$-Lie algebra on $K$ if and only if $\left<\F\right>\cap TK=\F$, where $\left<\F\right>$ is the smallest subalgebra of $\op{Diff}(K)$ containing $\F$. Hence, one could take this as a definition.
\end{Rem}

\paragraph{Iterated Jacobson Correspondence}

We extend the Theorem \ref{Jacobson} to power towers by simply iterating it. First, we need some definitions.

\begin{Def}
Let $W_\bullet$ be a power tower on $K$.
A sequence $T_{K/W_\bullet}\coloneqq(T_{K/W_1},T_{W_1/W_2},\ldots, T_{W_{n-1}/W_{n}},\ldots)$ is called an \emph{iterated relative tangent bundle} of $W_\bullet$.
\end{Def}

\begin{Def}\label{JacSeq}
A \emph{Jacobson sequence} on $K$  is a sequence of $p$-Lie algebras $(\F_1,\F_2,\ldots)$ such that $\F_1$ is a $p$-Lie algebra on $K$, and $\F_{i+1}$ is a $p$-Lie algebra on $\op{Ann}(\F_i)$ such that we have $\F_{i+1}\cap T_{\op{Ann}(\F_i)/K^{p^{i}}}=0$.

We say that a Jacobson sequence is of length at most $n$ if $\F_N=0$ for $N>n$. And, we say that it is of length $n$ if the $n$ is minimal such.
\end{Def}

\begin{Rem}
    The above condition $\F_{i+1}\cap T_{\op{Ann}(\F_i)/K^{p^{i}}}=0$ is equivalent to a map
    $d(\op{Ann}(\F_i)/K^{p^{i}}): \ \F_{i+1} \to TK^{p^{i}}\otimes \op{Ann}(\F_i)$ being injective, because $T_{\op{Ann}(\F_i)/K^{p^{i}}}=\op{ker}(d(\op{Ann}(\F_i)/K^{p^{i}}):TK\to T\op{Ann}(\F_i)\otimes K)$.
\end{Rem}

We can state the theorem.

\begin{Th}[Iterated Jacobson Correspondence]\label{part2}
Let $K$ be a field of characteristic $p>0$ such that $K/K^{p^\infty}$ is a finitely generated field extension. Let $n\ge 0$ be an integer.

There is a natural bijection between power towers on $K$ (of length at most $n$) and Jacobson sequences on $K$ (of length at most $n$).

Explicitly, this correspondence is given by
\begin{align*}
    W_\bullet \mapsto& T_{K/W_{\bullet}}\\
    (\F_1,\F_2,\ldots) \mapsto& \op{Ann}(\F_\bullet).
\end{align*}
Moreover, we have $\op{dim}_{W_{m}}(W_{m-1})=p^{\op{dim}_{\op{Ann}(\F_{m-1})}(\F_m)}$.
\end{Th}

\begin{proof}
    This moreover part is the moreover part from the Jacobson Correspondence \ref{Jacobson}. Also, the parts about preserving the lengths, and operations being inverse to each other are trivial. So, the only thing to show is that these operations are well defined.

    Let $W_\bullet$ be a power tower on $K$. We are going to show that its iterative relative tangent bundle is a Jacobson sequence. The first part of the definition is trivial, because, by the Jacobson Correspondence, we have
    $W_n=\op{Ann}(T_{W_{n-1}/W_n})$ for all $n\ge 1$ and  $T_{W_{n-1}/W_n}$ is a $p$-Lie algebra on $W_{n-1}$. The second part is to show that $T_{W_{n-1}/W_n}\cap T_{W_{n-1}/K^{p^{n-1}}}=0$. This follows again from the Jacobson Correspondence and the Definition \ref{powertower}. Indeed, let $\F=T_{W_{n-1}/W_n}\cap T_{W_{n-1}/K^{p^{n-1}}}$, then $\op{Ann}(\F)=W_n \cdot K^{p^{n-1}}=W_{n-1}$, so $\F=T_{W_{n-1}/W_{n-1}}=0$.

    Let $(\F_1,\F_2,\ldots)$ be a Jacobson sequence on $K$. We have that $K\supset \op{Ann}(\F_{1})\supset \op{Ann}(\F_{2})\supset \ldots$, and we need to show that $\op{Ann}(\F_{m})\cdot K^{p^{m-1}}=\op{Ann}(\F_{m-1})$. By the Jacobson Correspondence, we have that an equality $T_{\op{Ann}(\F_{m-1})/\op{Ann}(\F_{m})\cdot K^{p^{m-1}}}=T_{\op{Ann}(\F_{m-1})/\op{Ann}(\F_{m})}\cap T_{\op{Ann}(\F_{m-1})/K^{p^{m-1}}}=0$ holds, so we get $\op{Ann}(\F_{m-1})=\op{Ann}(\F_{m})\cdot K^{p^{m-1}}$.
\end{proof}

\begin{Rem}\label{origin.iterated}
    It is possible to discover the Theorem \ref{part2} by starting with a subfield $W$ of $K$ and trying to get the most from the Jacobson Correspondence. First, we compute $T_{K/W}$, and then we check that $\op{Ann}(T_{K/W})=W_1$. Now, we can observe that $K\supset W_1\supset W$, so we can repeat this procedure for $W_1 \supset W$. We get $T_{W_1/W}$, and
    $\op{Ann}(T_{W_1/W})=W\cdot W_1^p=W\cdot (W\cdot K^p)^p=W\cdot K^{p^2}=W_2$.
    This pattern continues and it leads us precisely to this theorem.
\end{Rem}

 Here are some examples.

\begin{Example}
An extension $K/K$ corresponds to $(0,0,\ldots)$.
\end{Example}

\begin{Example}
An extension $K/K^{p^n}$ corresponds to $(TK, TK^p,\ldots, TK^{p^{n-1}},0,0,\ldots)$.
\end{Example}

\begin{Example}
An extension $K/k$ corresponds to $(TK, TK^p,TK^{p^2},\ldots)$.
\end{Example}

\begin{Example}
Let $k$ be a perfect field. Let $K=k(x,y)$. Then a power tower of $k(x)$ on $ k(x,y)$ corresponds to a Jacobson sequence 
\[
(k(x,y)\frac{\partial}{\partial y},k(x,y^p)\frac{\partial}{\partial y^p},k(x,y^{p^2})\frac{\partial}{\partial y^{p^2}},\ldots).
\]
\end{Example}

\paragraph{The Main Theorem}

We prove our correspondence \ref{MainTheorem} by connecting the iterated Jacobson Correspondence \ref{part2} via an operation of unpacking \ref{part3} with the discrete Jacobson--Bourbaki Correspondence.

Let $K$ be a field of characteristic $p>0$ such that $K/k$ is a finitely generated field extension, where $k=K^{p^\infty}$.

We start with some definitions. First, we observe that a canonical filtration on differential operators, see the Definition \ref{canonical}, restricts well to all subalgebras.

\begin{Def}
Let $K\subset \mathcal{D}\subset \op{Diff}(K)$ be a $K$-subalgebra of $\op{Diff}(K)$. It admits a filtration by finite subalgebras $\mathcal{D}_n \coloneqq \mathcal{D} \cap \op{Diff}_{p^n}(K)$. We call it a canonical filtration of $\mathcal{D}$.

An algebra $\mathcal{D}$ is called of height at most $n$ if $\mathcal{D}=\mathcal{D}_n$. And, of height $n$, if the $n$ is minimal such.
\end{Def}

\begin{Def}
Let $W_\bullet$ be a power tower on $K$. An algebra of differential operators on $K$ relative to the tower $W_\bullet$ is defined by $\op{Diff}_{W_\bullet}(K)\coloneqq \bigcup_n \op{Diff}_{W_n}(K)$.
\end{Def}

Here are two reality checks.

\begin{Lem}\label{L1}
Let $W_\bullet$ be a power tower on $K$. Then $\op{Diff}_{W_\bullet}(K)_n=\op{Diff}_{W_n}(K)$.
\end{Lem}

\begin{proof}
It is enough to show $\op{Diff}_{W_N}(K)_n=\op{Diff}_{W_n}(K)$ for any $N\ge n$. But, this is equivalent to $\op{Diff}_{W_N}(K)\cap \op{Diff}_{p^n}(K)=\op{Diff}_{W_n}(K)$, what is nothing else but an equality $\op{End}_W(K)\cap \op{End}_{K^{p^n}}(K)=\op{End}_{W_n}(K)$, by the Lemma \ref{end=diff}.
\end{proof}

\begin{Lem}\label{L2}
Let $k\subset W\subset K$ be a subfield of $K$. Then we have that $\op{Diff}_W(K)=\op{Diff}_{W_\bullet}(K)$.
\end{Lem}

\begin{proof}
We look at a canonical filtration: $\op{Diff}_W(K)_n=\op{Diff}_W(K) \cap \op{Diff}_{K^{p^n}}(K)$. The right hand side, by the very construction, is $\op{Diff}_{W_n}(K)$. So, subalgebras $\op{Diff}_W(K), \op{Diff}_{W_\bullet}(K)$ have the same canonical filtrations. Therefore, they are equal.
\end{proof}

The following is a key corollary from the discrete Jacobson--Bourbaki Correspondence.

\begin{Cor}\label{effective}
Finite $K$-subalgebras of $\op{Diff}(K)$ naturally correspond to subfields of $K$ of finite exponents by operations $K\supset W\supset K^{p^n}\mapsto \op{Diff}_W(K)$, and $\mathcal{D}\mapsto \op{const}(\mathcal{D})\coloneqq \{x\in L; \forall_{D\in \mathcal{D}} [x,D]=0\}$.
\end{Cor}

\begin{proof}
    By the Theorem \ref{discrete}, extensions $K\supset W\supset K^{p^n}$ corresponds to an inclusion $\op{End}_W(K) \subset \op{End}_{K^{p^n}}(K)$ that, by the Lemma \ref{end=diff}, is $\op{Diff}_W(K) \subset \op{Diff}_{p^n}(K)\subset \op{Diff}(K)$.
\end{proof}

\begin{Rem}\label{groupoid}
    An alternative proof of the Corollary \ref{effective} goes by invocation of a theorem that finite locally free groupoids admit and are determined by their quotients, see 03BE at \cite{Stacks}. This is the way how Ekedahl approached a similar problem for varieties in his paper \cite{Ekedahl87}.
\end{Rem}

We extend the above to power towers.

\begin{Th}\label{part1}
Let $K$ be a field of characteristic $p>0$ such that $K/K^{p^\infty}$ is a finitely generated field extension. We put $k=K^{p^\infty}$. There is a natural correspondence between power towers on $K$, and $K$-subalgebras of differential operators on $K$ over $k$.

Explicitly,
this correspondence is given by the following formulas:
\begin{align*}
   W_\bullet\mapsto&  \op{Diff}_{W_\bullet}(K)\subset  \op{Diff}_k(K)\\
   K\subset \mathcal{D}\subset \op{Diff}_k(K) \mapsto& \op{const}(\mathcal{D}_\bullet)\coloneqq \{x\in L; \forall_{D\in \mathcal{D}_\bullet} [x,D]=0\}.
\end{align*}
Moreover, $\op{dim}_K(\mathcal{D}_i)=\op{dim}_{W_i}(K)$.
\end{Th}

\begin{proof}
    Any subalgebra $\mathcal{D}$ is determined by its canonical filtration, and we have $\mathcal{D}_1\subset \mathcal{D}_2\subset \ldots \subset \mathcal{D}$. 
    By the Corollary \ref{effective}, there are subfields $W_i$ corresponding to $\mathcal{D}_i$ such that we have
    $\mathcal{D}_i=\op{Diff}_{W_i}(K)$. However, by the Lemmas \ref{L1} and \ref{L2}, it means that $W_\bullet$ is a power tower on $K$. 
    Therefore, a canonical filtration determines a power tower on $K$, and this power tower determines the canonical filtration on the algebra of differential operators relative to it. This finishes the proof.
\end{proof}

The iterated Jacobson Correspondence \ref{part2} connects with the above proposition via an operation from the following proposition. We call this operation an unpacking. It produces a Jacobson sequence out of a subalgebra of differential operators.

\begin{Prop}[Unpacking]\label{part3}
Let $W_\bullet$ be a power tower on $K$. Let $(\F_i)$ be an iterated relative tangent bundle of $W_\bullet$, and let $\mathcal{D}=\op{Diff}_{W_\bullet}(K)$. Then, for $i\ge 1$, we have
\[
d(K/W_{i-1})(\mathcal{D})\cap TW_{i-1} = \F_i.
\]
In particular, the first degree of $\mathcal{D}$, i.e. $\mathcal{D}\cap TK$, is equal $T_{K/W_1}$.
\end{Prop}

\begin{proof}
We start with the first degree. Observe that $\left<T_{K/W_1}\right>=\mathcal{D}_1$ by the Corollary \ref{effective} and the Theorem \ref{Jacobson}, because both correspond to the same object, i.e. $W_1$. So, by the Remark \ref{<L>capTK=L}, we have $\mathcal{D}\cap TK =\mathcal{D}_1\cap TK=T_{K/W_1}=\F_1$.

Let $i> 1$. First, we observe that the intersection makes sense, because we can inject $W$-space $TW_{i-1}$ into $\op{Diff}(W_{i-1})\otimes K$. Second, we observe that $d(K/W_{i-1})(\mathcal{D})\cap TW_{i-1}=d(K/W_{i-1})(\mathcal{D})\cap \left<TW_{i-1}\right>\otimes K\cap TW_{i-1}=d(K/W_{i-1})(\mathcal{D}_i)\cap TW_{i-1}$. Therefore, we can work with $\mathcal{D}_i$. 
By the Lemma \ref{end=diff}, it is $\op{End}_{W_i}(K)$. By the Proposition \ref{differential}, $d(K/W_{i-1})(\op{End}_{W_i}(K))$ is equal $\op{End}_{W_i}(K)$ restricted to $W_{i-1}$, but this is precisely, as an algebra of operators, $\op{End}_{W_i}(W_{i-1})\otimes K$. Therefore, $\op{End}_{W_i}(W_{i-1})\otimes K \cap TW_{i-1}=T_{W_i/W_{i-1}}=\F_i$.
\end{proof}

Here is the main theorem of this paper.

\begin{Th}\label{MainTheorem}
Let $K$ be a field of characteristic $p>0$. Let $k\coloneqq \bigcap_{n\ge 0} K^{p^n}$. We assume that $K/k$ is a finitely generated field extension.

There are natural bijections between the following data on $K$:
\begin{itemize}
    \item power towers on $K$, see the Definition \ref{powertower},
    \item Jacobson sequences on $K$, see the Definition \ref{JacSeq},
    \item $K$-subalgebras of differential operators on $K$: $K\subset \mathcal{D}\subset \op{Diff}_k(K)$
\end{itemize}
Explicitly,
these bijections are given by the following operations:
\begin{itemize}
    \item an iterated tangent bundle: $W_\bullet \mapsto (T_{W_0/W_1},T_{W_1/W_2},T_{W_2/W_3}, \ldots)$,
    \item an iterated $\op{Ann}(\bullet)$: $(\F_1,\F_2,\ldots)\mapsto \op{Ann}(\F_\bullet)$,
    \item an algebra of relative differential operators: $W_\bullet \mapsto \op{Diff}_{W_\bullet}(K)$,
    \item an unpacking: $\mathcal{D} \mapsto d(K/W_{n-1})(\mathcal{D})\cap TW_{n-1} = \F_n$.
\end{itemize}
Moreover, we have that ranks $\op{rk}(\F_i)\coloneqq \op{dim}_{W_{i-1}}(\F_i)$ are finite and nonincreasing: $\op{rk}(\F_i)\ge \op{rk}(\F_{i+1})$, and they satisfy $p^{\op{rk}(\F_i)}=\op{dim}_{W_{i+1}}(W_i)$.
\end{Th}

\begin{proof}
    The theorem is an amalgamation of the following results \ref{part1}, \ref{part2}, \ref{part3}, \ref{NonincreasingOrder}.
\end{proof}

Here are some examples.

\begin{Example}
A subalgebra $K$ corresponds to $K$.
\end{Example}

\begin{Example}
A subalgebra $\op{Diff}(K)$ corresponds to $k$.
\end{Example}

\begin{Example}
Let $k$ be a perfect field. Let $K=k(x,y)$. Then a power tower of $k(x)\subset k(x,y)$ corresponds to a subalgebra $K$-spanned by $\frac{1}{a!}\frac{\partial^{a}}{\partial y^{a}}$ for $a\ge 0$.
\end{Example}

More examples can be found in the Example \ref{examples}.

We finish with a remark comparing the above theorem with the Theorem \ref{JB}.

\begin{Rem}\label{closures}
Let $K$ be like in the Theorem \ref{MainTheorem}. Let $\mathcal{D}\subset \op{Diff}(K)$ be a $K$-subalgebra. Let $W_\infty$ be first integrals of this algebra.

We can consider $\mathcal{D}$ as a subalgebra of $\op{End}_\Z(K)$. Let $\overline{\mathcal{D}}$ be a closure of $\mathcal{D}$ in $\op{End}_\Z(K)$ in a finite topology. This closure is of the form $\op{End}_W(K)$, where $W$ is a subfield of $K$ by the Theorem \ref{JB}. Actually, $W=W_\infty$. Indeed, $\op{const}(\mathcal{D})=W_\infty$, so $\op{End}_{W_\infty}(K)\supset\mathcal{D}$. Hence, $\op{End}_W(K)\subset \op{End}_{W_\infty}(K)$, so $W_\infty\subset W$. But, $\op{End}_{W}(K)\supset\mathcal{D}$, so $\op{const}(\op{End}_{W}(K))=W\subset \op{const}(\mathcal{D})=W_\infty$. 
However, $\mathcal{D}$ is determined by $W_\infty$ if and only if a corresponding to it power tower is algebraically integrable, and this may not be the case, e.g. \ref{NonAlgInt} and \ref{examples}.

Furthermore, one could say that there are more not algebraically integrable power towers than algebraically integrable ones. Indeed, let $k$ be a finite field, then if $K/k$ has a positive transcendental degree, then there are only countable many subfields between $k$ and $K$, but there are uncountable many power towers on $K$. Therefore, in almost all cases, a closure of a subalgebra of differential operators says little about this subalgebra.
\end{Rem}

\newpage
\section{Applications}\label{section5}
In this section, we demonstrate some results by applying knowledge from the previous sections. There are three paragraphs here. 
In the first one \ref{injection}, we prove that separably closed subfields inject into power towers \ref{infty=s}. 
In the second one \ref{extension}, we discuss a problem of extending $1$-foliations to $2$-foliations. We describe what is an obstruction for this \ref{Splittings}, and that the obstruction is always trivial \ref{always}. In the last one \ref{examples}, we compute more examples of power towers on $k(x,y)$ \ref{examples}.

\paragraph{Injection of Subfields into Power Towers}\label{injection}

\begin{Prop}\label{infty=s}
Let $K$ be a field of characteristic $p>0$ such that $K/K^{p^\infty}$ is a finitely generated field extension. We put $k=K^{p^\infty}$.
    
Let $W$ be a subfield of $K$. Then a field of first integrals of a power tower of $W$ on $K$ is a separable closure of $k\cdot W$ in $K$, i.e. $W_\infty=\left(k\cdot W\right)^s$. Moreover, the power tower of $W$ on $K$ and a power tower of $W_\infty$ on $K$ are equal.
\end{Prop}

\begin{proof}
We can assume that $k\subset W$.

First, observe that $(W^s)_\bullet=W_\bullet$. Indeed, for every $n\ge 0$, we have $K \supset W^s_n\supset W_n \supset K^{p^n}$, so $W^s_n\supset W_n$ is purely inseparable. But, on the other hand, $W^s_n\supset W_n$ is separable. Therefore, we have $W_n^s=W_n$ .

We conclude that $W^s\subset W_n$ for every $n$, hence $W^s\subset W_\infty$.

We also have $(W_\infty)^s = W_\infty$. Indeed, we have $(W_\infty)_n\subset W_n$, so $(W_\infty)^s\subset W_\infty$. Therefore, $(W_\infty)^s=W_\infty$.

Moreover, power towers of $W^s$ and $W_\infty$ on $K$ are equal since $W^s \subset W_\infty \subset W_n$ gives $W_n=W^s\cdot K^{p^n} \subset W_\infty\cdot K^{p^n} \subset W_n$.

Now, we know that we have an inclusion $W^s\subset W_\infty$, and that both of these fields are separably closed in $K$. Moreover,  it is a finitely generated extension of fields. If the extension was nontrivial, then we would have $\Omega_{W_\infty/W}\ne  0$. Indeed, if it was nontrivial, then we could find a finite sequence of elements $x_i\in W_\infty$, where $i=1,\ldots, r$, such that
\[
L_0\coloneqq W\subset L_1\coloneqq W(x_1)^s \subset L_2\coloneqq L_1(x_2)^s \subset ... \subset L_r\coloneqq L_{r-1}(x_r)^s=W_\infty
\]
is a sequence of proper extensions, and $x_i$ is purely inseparable, or purely transcendental over $L_{i-1}$. A simple calculation shows that $\Omega_{L_r/L_{r-1}}$ has dimension one, so it is non-zero. Furthermore, $\Omega_{L_r/L_0}$ surjects onto it, so it is non-zero either.

Hence, still assuming that the extension $W_\infty/W^s$ is nontrivial, we can conclude that $\op{Diff}_W(W_\infty)\ne W_\infty$, because the algebra contains the dual of $\Omega_{W_\infty/W}$. However, by the Corollary \ref{short exact sequence for diff} restricted to $W$-linear operators, we have
\[
0\to \op{Diff}_W(K) \E_W(K/W_\infty)\to \op{Diff}_W (K) \to \op{Diff}_W (W_\infty) \otimes K\to 0.
\]
And, by the Theorem \ref{MainTheorem}, the Lemma \ref{L2}, and the equality of power towers of $W$ and $W_\infty$, we have that $\op{Diff}_W(K)=\op{Diff}_{W_\infty}(K)$. In particular, we have $\E_W(K/W_\infty)=\E_W(K/W)$. However, this means that $\E_W(W_\infty/W)=0$, what is the contradiction with the fact that $\op{Diff}_W(W_\infty)\ne W_\infty$.
\end{proof}

The following extends the Corollary \ref{FiniteLength}.

\begin{Cor}\label{Recovery}
Let $K$ be a field of characteristic $p>0$ such that $K/K^{p^\infty}$ is a finitely generated field extension. We put $k=K^{p^\infty}$. We have the following diagram
\begin{center}
    \begin{tikzcd}
        \{K\supset W\supset K^{p^n}\}\ar[r, hook]&\{K\supset W=W^s\supset k\} \ar[r, hook, "W\mapsto W_\bullet"]& \text{Power Towers on $K$}\ar[l, bend right = 35, "W_\bullet \mapsto W_\infty"'],
    \end{tikzcd}
\end{center}
where $n$ is a positive integer, $W^s$ is a separable closure of $W$ in $K$, the Lemma \ref{separableclosure}, the first arrow is a map from subfields of finite exponent $\le n$ to subfields containing $k$ that are separably closed in $K$. Moreover, the operation of taking first integrals, the Definition \ref{W.infty}, splits the arrow of taking power towers here, the Definition \ref{PTsubfield}, hence the operation  $W\mapsto W_\infty$ is injective on these subfields.
\end{Cor}

\begin{proof}
    The upmost left map is an inclusion, because subfields of finite exponent are separably closed in $K$. Indeed, let $W^s$ be a separable closure of a subfield $W$ of exponent $\le n$, see the Lemma \ref{separableclosure}, then we have that $K^{p^n}\subset W\subset W^s \subset K$. However, this says that the extension $W\subset W^s$ is a separable subextension of a purely inseparable extension $W\subset K$, so it must be purely inseparable as well. This means that $W=W^s$.

    The splitting is precisely the Proposition \ref{infty=s}. It says that $W_\infty=(W\cdot k)^s$. However, here, we have that $W\cdot k=W$, and $W=W^s$. Therefore, we get $W_\infty=W$.
\end{proof}

\paragraph{Extension Problem}\label{extension}
Our $n$-foliations, the Definition \ref{n-foliation}, are analogs of Ekedahl's $n$-foliations from his paper \cite{Ekedahl87}. In this paper, Ekedahl considers a problem of extending a $1$-foliation to a $2$-foliation. His main result regarding this is an example of a $1$-foliation on a variety that does not admit any extension to a $2$-foliation on the variety, it is on the page 145 there. This example is important, because it is also a positive characteristic counterexample to a lemma used in a proof of a Miyaoka's semipositivity theorem, the Example 8.8 in \cite{Mi87}. However, it is not yet resolved in the literature if being extendable to a $\infty$-foliation on a variety is the only obstruction for $1$-foliations on varieties to satisfy that lemma. Therefore, here, we approach this extension problem for fields as a small step towards a resolution.

First, we define the problem.

\begin{Problem}
    Let $K$ be a field of characteristic $p>0$. Let $W_1$ be a subfield of $K$ of exponent $1$, i.e. $K\supset W_1\supset K^p$ and $W_1\ne K$. 
    Is there a subfield $W_2$ such that $W_1,W_2$ is a $2$-foliation? Furthermore, are there subfields $W_n$ for $n\ge 2$ such that $W_\bullet$ is a $\infty$-foliation?
\end{Problem}

By the following lemma, to answer the second question from the above problem it is enough to answer the first one.

\begin{Lem}
Let $K$ be a field of characteristic $p>0$. 
Let $W_1,W_2$ be a $2$-foliation on $K$, and let $W_2,W_3$ be a $2$-foliation on $W_1$. Then $W_1,W_2,W_3$ is a $3$-foliation on $K$.
\end{Lem}

\begin{proof}
    It is enough to prove that $W_3\cdot K^{p^2}=W_2$. We have that $W_1=W_2\cdot K^p$, so $W_1^p=W_2^p\cdot K^{p^2}$. Consequently, we have $W_2=W_3\cdot W_1^p=W_3\cdot (W_2^p\cdot K^{p^2})=(W_3\cdot W_2^p)\cdot K^{p^2}=W_3\cdot K^{p^2}$, because $W_3\supset W_2^p$.
\end{proof}

Second, we can write down an obstruction in terms of Jacobson sequences \ref{JacSeq}. In the below, an abstract $p$-Lie algebra on a field $K$ is a $K$-vector space enriched with a Lie bracket, and a $p$-power operation satisfying some axioms to make them behave like $p$-Lie subalgebras of $TK$. It is made precise in a paper \cite{p-lie-coho}.

\begin{Prop}\label{Splittings}
Let $K$ be a field of characteristic $p>0$ such that $K/K^{p^\infty}$ is a finitely generated field extension.

Let $W_1$ be a subfield of $K$ of exponent one. Then we have a short exact sequence of abstract $p$-Lie algebras on $W_1$:
\[
0 \to T_{W_1/K^p} \to TW_1 \xrightarrow{d(W_1/K^p)} T_{K^P/W_1^p}\otimes W_1 \to 0.
\]
This sequence admits a right split of abstract $p$-Lie algebras if and only if $W_1$ can be extended to a $2$-foliation on $K$. Moreover, all such splittings are in a bijection with all such extensions. Explicitly, if a splitting $\sigma$ corresponds to an extension $W_2$, then the image of $\sigma$ is $T_{W_1/W_2}$.
\end{Prop}

\begin{proof}
    The exactness of the sequence follows from the Jacobson Correspondence \ref{Jacobson} applied to extensions $W_1 \supset K^p\supset W^p_1$.

    Any extension of $W_1$ to a $2$-foliation is a pair of fields $W_1,W_2$. By the Theorem \ref{MainTheorem}, it is equivalent to a pair of $p$-Lie algebras $T_{K/W_1},T_{W_1/W_2}$ that is a Jacobson sequence \ref{JacSeq}, so $T_{W_1/W_2}\subset TW_1$ and $T_{W_1/W_2}\cap T_{W_1/K^p}=0$. Therefore, any such $T_{W_1/W_2}$ defines a right splitting that preserves Lie brackets, and $p$-powers. And, any such splitting provides such subspace as its image.
\end{proof}

Finally, we show that the obstruction is trivial.

\begin{Prop}\label{always}
    Let $K$ be like in the Proposition \ref{Splittings}.
    
    Let $W_1$ be a $1$-foliation on $K$. Then $W_1$ can be extended to a $2$-foliation on $K$.
\end{Prop}

\begin{proof}
    Let $r$ be a rank of a $1$-foliation $W_1$. Put $k=K^{p^{\infty}}$. Let $\op{tr.deg}(K/k)=N$. 

    In the proof, we use a lot of facts about $p$-bases, see the Definition \ref{Coordinates} and its references.

    There are $r$-elements $t_1,t_2,\ldots,t_r\in K$ such that $W_1(t_\bullet)=K$. These elements are $p$-independent in $K$, because monomials $\prod_i t_i^{a_i}$, where $0\le a_i<p$, are $W_1$-linearly independent. Therefore, they are $K^p$-linearly independent, because $W_1\supset K^p$. Moreover, their $p$-powers $t_i^p\in W_1$ are $p$-independent in $W_1$, because this means that monomials monomials $\prod_i t_i^{p \cdot a_i}$, where $0\le a_i<p$, are $W_1^p$-linearly independent, which, by taking $p$-roots, is the above information.

    Let $a_1,a_2,\ldots, a_N$ be a $p$-basis for $W_1/k$. This is equivalent to $da_\bullet$ being a $W_1$-linear basis of $\Omega_{W_1/k}$.
    Consequently, it admits an exchange lemma, i.e. any $p$-independent set can be completed to an $p$-independent set of maximal size by adding some elements from other maximal size independent set. Therefore, there are $N-r$ elements among $a_i$, say $a_1,a_2,\ldots,a_{N-r}$, such that $a_1,a_2,\ldots,a_{N-r}, t_1^p, t_2^p, \ldots, t_r^p$ is a $p$-basis for $W_1/k$.

    We claim that $a_1,a_2,\ldots,a_{N-r}, t_1, t_2, \ldots, t_r$ is a $p$-basis for $K/k$. Indeed, they are algebraically independent, because if not, then there would be a relation $F$ between them, but then $F^p$ would be such relation for $a_1,a_2,\ldots,a_{N-r}, t_1^p, t_2^p, \ldots, t_r^p$, a contradiction.
    Next, we have a diagram
    \begin{center}
        \begin{tikzcd}
        &K=W_1(t_\bullet)&\\
        k(a_\bullet,t_\bullet)^s\ar[ur]&&W_1\ar[ul]\\
        &k(a_\bullet,t^p_\bullet)^s\ar[ur, equal]\ar[ul]&
        \end{tikzcd}
    \end{center}
    wherein the separable closures are taken in $K$.
    The extensions $k(a_\bullet,t_\bullet)^s/k(a_\bullet,t^p_\bullet)^s$ and $K/W_1$ are of order $p^r$, so the extension $K/k(a_\bullet,t^p_\bullet)^s$ must be trivial, i.e. $k(a_\bullet,t^p_\bullet)^s=K$. Therefore, $a_1,a_2,\ldots,a_{N-r}, t_1, t_2, \ldots, t_r$ is a separating transcendental basis, so it is a $p$-basis for $K/k$.

    Finally, consider the following diagram of fields:

    \begin{center}
        \begin{tikzcd}
            K=k(a_\bullet,t_\bullet)^s & W_1=k(a_\bullet,t^p_\bullet)^s\ar[l] & W_2\coloneqq k(a_\bullet,t^{p^2}_\bullet)^s\ar[l]\\
            k(a_\bullet,t_\bullet)\ar[u]&k(a_\bullet,t^p_\bullet)\ar[u]\ar[l]&k(a_\bullet,t^{p^2}_\bullet)\ar[u]\ar[l]
        \end{tikzcd}
    \end{center}
    We claim that $K\supset W_1\supset W_2$ is a $2$-foliation.
    Indeed, we have that $W_1\supset K^{p}$, $W_2\supset K^{p^2}$, $[K:W_1]=[W_1:W_2]=p^r$, so the only remaining thing to check is $W_2\cdot K^p=W_1$.

    It is easy: 
    \[
    W_1\supset W_2\cdot K^p\supset k(a_1,t^{p^2})^s(a_1,t^p)=k(a_1,t^{p^2})^s(t^p)=k(a_1^p,t^p)^s=W_1
    \]
    where $k(a_1,t^{p^2})^s(t^p)=k(a_1^p,t^p)^s$, because
    $k(a_1,t^{p^2})^s(t^p)=(k(a_1,t^{p^2})^s(t^p))^s=k(a_1,t^p)^s$, where the second equality follows from the subfield $k(a_1,t^{p^2})^s(t^p)$ being of finite exponent since this implies being separably closed in $K$.
\end{proof}

\paragraph{Examples}\label{examples}
We present a simple class of power towers on $k(x,y)$. In the below example, we use symbols $\frac{1}{p^{m}!}\frac{\partial^{p^m}}{\partial y ^{p^m}}$ and $\frac{\partial}{\partial \left(y^{p^m}\right)}$, and $\frac{1}{p^{m}!}\frac{\partial^{p^m}}{\partial x ^{p^m}}$ and $\frac{\partial}{\partial \left(x^{p^m}\right)}$ interchangeably.

\begin{Example}\label{examples}
Let $K=k(x,y)$ be rational functions in two variables over a perfect field $k$ of characteristic $p>0$. 

We define a power tower of a ``subfield'' 
\[
W_\infty=k(x+A_1y^p+A_2y^{p^2}+A_3y^{p^3}+\ldots),
\]
on $K$, where $A_i\in k$, by the following formulas
\[
W_n \coloneqq k(x+A_1 y^p + A_2 y^{p^2} + \ldots, A_{n-1} y^{p^{n-1}},y^{p^n}).
\]
It is clear from the definition that $W_\bullet$ is a power tower on $K$. Moreover, it is a $\infty$-foliation of rank $1$. Indeed, we have that $W_n(y^{p^{n-1}})=W_{n-1}$, and $W_n\ne W_{n-1}$ for $n\ge 1$.

We can compute Jacobson sequences and subalgebras of differential operators corresponding to these examples, see the Theorem \ref{MainTheorem}. We start with the Jacobson sequences.

The first extension $K/W_1$ is $k(x,y)\supset k(x,y^p)$. This corresponds to a $p$-Lie algebra on $K$ generated by a vector field $\frac{\partial}{\partial y}$. 

The second extension $W_1/W_2$ is $k(x,y^p)\supset k(x+A_1y^p,y^{p^2})$.  The tangent space of the bigger field is still easy to compute in terms of $x,y$. Indeed, it is spanned by $\frac{\partial}{\partial y^p}$ and $\frac{\partial}{\partial x}$. Clearly, this extension corresponds to a $p$-Lie algebra on $W_1$ generated by a vector field $\frac{\partial}{\partial y^p}-A_1\frac{\partial}{\partial x}$. To check that it is a $p$-Lie algebra we use the Corollary \ref{obvious.values} and the assumption $A_1\in k$.

The third extension $W_2/W_3$ is $k(x+A_1y^p,y^{p^2})/k(x+A_1y^p+A_2y^{p^2}, y^{p^3})$. Starting from here, the computation of a tangent space in terms of $x,y$ gets harder. We use the fact that a differential $d(K/W_2)$ is surjective, the Proposition \ref{differential}, to find differential operators on $K$ that are mapped to a basis of $TW_2$. 
It is easy to check that operators $\frac{\partial}{\partial y^{p^2}}-A_1^p\frac{\partial}{\partial x^p}$ and $\frac{\partial}{\partial x}$ are mapped to $TW_2$ and span it. In this basis, the extension corresponds to a vector field $\frac{\partial}{\partial y^{p^2}}-A_1^p\frac{\partial}{\partial x^p} - A_2\frac{\partial}{\partial x}$.

Let $n\ge 3$. We have that $TW_{n-1}$ is spanned by operators
$\frac{\partial}{\partial y^{p^{n-1}}}-A_1^{p^{n-2}}\frac{\partial}{\partial x^{p^{n-2}}} - A_2^{p^{n-3}}\frac{\partial}{\partial x^{p^{n-3}}}-\ldots-A_{n-2}^p\frac{\partial}{\partial x^{p}}$ and $\frac{\partial}{\partial x}$. And an extension $W_{n-1}/W_{n}$ corresponds to a vector field $\frac{\partial}{\partial y^{p^{n-1}}}-A_1^{p^{n-2}}\frac{\partial}{\partial x^{p^{n-2}}} - A_2^{p^{n-3}}\frac{\partial}{\partial x^{p^{n-3}}}-\ldots-A_{n-1}\frac{\partial}{\partial x}$.

Finally, by an operation of unpacking, we can conclude that $\op{Diff}_{W_\bullet}(K)$ is generated by operators
$\frac{\partial}{\partial y}, \frac{\partial}{\partial y^p}-A_1\frac{\partial}{\partial x},\ldots, \frac{\partial}{\partial y^{p^{n-1}}}-A_1^{p^{n-2}}\frac{\partial}{\partial x^{p^{n-2}}} - A_2^{p^{n-3}}\frac{\partial}{\partial x^{p^{n-3}}}-\ldots-A_{n-1}\frac{\partial}{\partial x},\ldots$.

\end{Example}

\end{document}